\documentclass[a4paper,reqno,11pt]{amsart}
\usepackage[T1]{fontenc}
\usepackage[utf8]{inputenc}
\usepackage[english]{babel}
\usepackage[dvipsnames,svgnames,table]{xcolor}
\usepackage{amssymb, amsmath, amsthm, graphicx, enumerate, color, mathtools, comment, caption, subcaption, float, lmodern}
\usepackage[foot]{amsaddr}
\usepackage[shortlabels]{enumitem}
\usepackage[unicode=true]{hyperref}
\usepackage[capitalise, noabbrev]{cleveref}

\usepackage{textpos}
\hypersetup{
    pdftitle={Counterexamples to Statements on Isometric Graph Coverings},
    colorlinks,
    linkcolor={RoyalBlue},
    citecolor={RubineRed},
    urlcolor={blue!80!black}
}
\crefformat{subsection}{Subsection #2#1#3}

\makeatletter
\def\thm@space@setup{
  \thm@preskip=4mm
  \thm@postskip=0mm
}
\makeatother

\theoremstyle{plain} 
\newtheorem{theorem}{Theorem}
\newtheorem{lemma}[theorem]{Lemma}
\newtheorem{obs}[theorem]{Observation}

\newtheorem{corollary}[theorem]{Corollary}
\newtheorem{question}{Question} 
\Crefname{question}{Question}{Questions}
\newtheorem{metaquestion}[question]{(Meta) Question} 
 
\crefname{metaquestion}{Question}{Questions}
\Crefname{openquestion}{Question}{Questions}
\theoremstyle{remark}
\newtheorem{claim}{Claim}[theorem]

\newenvironment{proofclaim}[1][]
	{\vspace{-\topsep}\begin{proof}[Proof] }{\end{proof}}

\newcommand{\calC}{\mathcal{C}}

\newcommand{\calE}{\mathcal{E}}

\newcommand{\calH}{\mathcal{H}}

\newcommand{\calT}{\mathcal{T}}

\newcommand{\q}[1]{``#1''}

\newcommand{\N}{\mathbb{N}}

\newcommand{\ext}{\mathcal E}

\let\le\leqslant
\let\ge\geqslant
\let\leq\leqslant
\let\geq\geqslant

\let\subset\subseteq

\let\epsilon\varepsilon

\let\setminus\backslash

\newcommand{\para}[1]{\mathbf{#1}}
\DeclareMathOperator\tw{\para{tw}}
\DeclareMathOperator\pw{\para{pw}}
\DeclareMathOperator\td{\para{td}}

\DeclareMathOperator\dist{dist}
\newcommand{\Apex}{\mathbb{A}}  
\DeclareMathOperator\Inc{Inc}

\DeclareMathOperator\ball{Ball}

\renewenvironment{enumerate}{\begin{enumorig}[label=\textup{(\roman*)}, noitemsep, 
topsep=2pt plus 2pt, labelindent=.2em, leftmargin=*, widest=iii]}{\end{enumorig}}

\newenvironment{enumerateAlpha}{\begin{enumorig}[label=\textup{(\alph*)}, noitemsep, 
topsep=2pt plus 2pt, labelindent=.2em, leftmargin=*, widest=iii]}{\end{enumorig}}

\newcommand{\defin}[1]{\emph{\textcolor{ForestGreen}{#1}}}
\newcommand{\defmath}[1]{\defin{\text{$#1$}}}

\newcommand{\jedrzej}[1]{\textcolor{ProcessBlue}{Jędrzej: #1}}

\newcommand{\paul}[1]{\textcolor{purple}{[Paul: #1]}}

\newcommand{\weichan}[1]{\textcolor{red}{Weichan: #1}}

\newcommand{\julien}[1]{\textcolor{orange}{Julien: #1}}

\begin{document} 
\title[Covering graphs by isometric trees] 
{Covering graphs by isometric trees}

\email{\href{mailto:paul.bastide@ens-rennes.fr}{paul.bastide@ens-rennes.fr}}
\email{\href{mailto:julien.duron@ens-lyon.fr}{julien.duron@ens-lyon.fr}}
\email{\href{mailto:jedrzej.hodor@gmail.com}{jedrzej.hodor@gmail.com}}
\email{\href{mailto:wcliu@sdu.edu.cn}{wcliu@sdu.edu.cn}}
\email{\href{mailto:xiangxiangnie@sdu.edu.cn}{xiangxiangnie@sdu.edu.cn}}

\author[Bastide]{Paul Bastide}
\address[Bastide]{Mathematical Institute, University of Oxford, Oxford, United Kingdom}

\author[Duron]{Julien Duron}
\address[Duron]{Institute of Informatics, University of Warsaw, Poland}

\author[Hodor]{Jędrzej Hodor}
\address[Hodor]{Theoretical Computer Science Department, 
Faculty of Mathematics and Computer Science and  Doctoral School of Exact and Natural Sciences, Jagiellonian University, Krak\'ow, Poland}

\thanks{P. Bastide is supported by ERC Advanced Grant
883810. 
J. Duron is supported by the project BOBR that
has received funding from the European Research Council (ERC) under
the European Union’s Horizon 2020 research and innovation programme
(grant agreement No 948057).
J.\ Hodor is supported by a Polish Ministry of Education and Science grant (Perły Nauki; PN/01/0265/2022).
W. Liu is supported by the Postdoctoral Fellowship Program of CPSF (No.\,GZC20252020), the China Postdoctoral Science
Foundation (No.\,2025M783118) and partially supported by National Natural Science Foundation of China (No.\,12661075).}

\author[Liu]{Weichan Liu}
\address[Liu]{School of Mathematics, Shandong University, Jinan, China}

\author[Nie]{Xiangxiang Nie}
\address[Nie]{Data Science Institute, Shandong University, Jinan, China}


\begin{abstract}
    A connected subgraph of a graph is isometric if it preserves distances.
    Recently, graphs admitting a vertex or edge covering by a small number of isometric paths have been studied. 
    In this paper, we consider the analogous problem for isometric trees, focusing on the treewidth of graphs admitting a vertex or edge covering by a small number of such trees. 
    Baste, De Meyer, Giocanti, Objois, and Picavet showed that for coverings by two isometric trees, the treewidth is bounded. 
    We show that already for three isometric trees, the treewidth can be linear in the number of vertices. 
    On the positive side, we show that for graphs of bounded degree coverable by a small number of isometric trees, the treewidth is sublinear in the number of vertices.
\end{abstract}

 \begin{textblock}{20}(-1.75, 10.6)
 \includegraphics[width=40px]{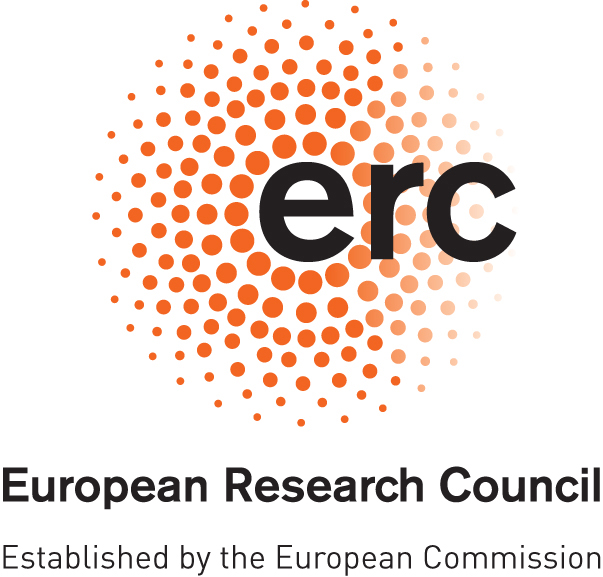}%
 \end{textblock}
 \begin{textblock}{20}(-1.75, 11.6)
 \includegraphics[width=40px]{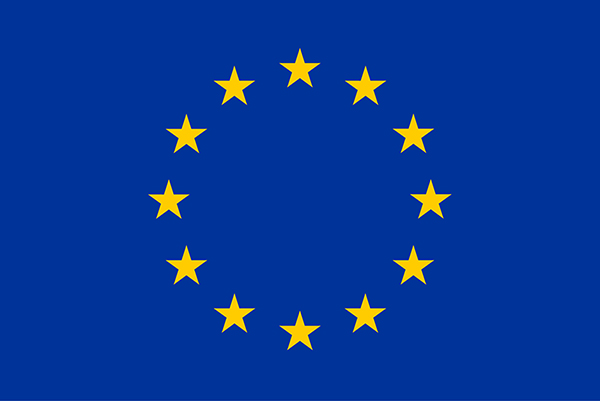}%
 \end{textblock}

\maketitle

\section{Introduction}\label{sec:introduction}
The \defin{distance} between two vertices $u$ and $v$ in a connected graph $G$ is the minimum length of a path connecting $u$ and $v$ (i.e.\ its number of edges); we denote this value by \defin{$\dist_G(u,v)$}.
A connected subgraph $H$ of a graph $G$ is \defin{isometric} if it preserves distances, that is, for any two vertices $u$ and $v$ in $H$, we have $\dist_H(u,v) = \dist_G(u,v)$.
We say that a set of subgraphs $\calH$ of a graph $G$ \defin{edge-covers} $G$ when $E(G) = \bigcup_{H \in \calH} E(H)$, and \defin{vertex-covers} $G$ when $V(G) = \bigcup_{H \in \calH} V(H)$.
When all subgraphs in $\calH$ are isometric in $G$, we say that $\calH$ is an \defin{isometric edge-covering} of $G$, and \defin{isometric vertex-covering} of $G$, respectively.
With a slight abuse of notation, sometimes implicitly by \q{subgraph}, we mean \q{isomorphic to a subgraph}.

The general question we are interested in is the following.

\begin{metaquestion}\label{main-question}
Given a positive integer $k$ assume that a graph $G$ is edge-covered by connected isometric subgraphs $H_1,\dots,H_k$.
Assuming \q{some} property of $H_1,\dots,H_k$, can we infer \q{some} (potentially different) property of $G$?
\end{metaquestion}

In our considerations, \q{some} property will be a bound on one of the classical structural graph parameters.
For a graph $G$, the \defin{treewidth}, \defin{pathwidth}, and \defin{treedepth} of $G$ are respectively denoted by \defin{$\tw(G)$}, \defin{$\pw(G)$}, and \defin{$\td(G)$}.
See~\cref{sec:preliminaries} for the definitions.
We are also particularly interested in the case where the covering graphs are trees.

\Cref{main-question} is motivated by a result of Dumas, Foucaud, Perez, and Todinca~\cite{Dumas}.
They proved that there exists a function $f$ such that if a graph $G$ can be edge-covered by $k$ isometric paths, then $\pw(G) \leq f(k)$.
See also a \q{coarse} variant of this statement of Hatzel and Pilipczuk~\cite{chubby}.
Baste, De Meyer, Giocanti, Objois, and Picavet~\cite{Baste} later improved the bound on $f$ from exponential to polynomial.
They also asked if an analogous result holds for trees and treewidth.
\begin{question}\cite{Baste}\label{question-trees-treewidth}
Does there exist a function $g$ such that if a graph $G$ can be edge-covered by $k$ isometric trees, then $\tw(G) \leq g(k)$?
\end{question}
As evidence, Baste et al.\ proved that this holds for $k=2$ in a strong sense: if a graph $G$ can be edge-covered by two isometric trees, then $\tw(G) \leq 2$.
A natural generalization of~\Cref{question-trees-treewidth} is the following.

\begin{question}\label{question-parameters}
Let $\para{p},\para{q} \in \{\tw,\pw,\td\}$. Does there exist a function $h$ such that for every positive integer $c$ if a graph $G$ can be edge-covered by $k$ connected isometric subgraphs with the parameter $\para{p}$ at most $c$, then $\para{q}(G) \leq h(k,c)$?
\end{question}

Recall that for every graph $H$, we have $\tw(H) \leq \pw(H) \leq \td(H)-1$.
Therefore, the weakest variant of~\Cref{question-parameters} is when $\para{p} = \td$ and $\para{q} = \tw$.
We answer~\Cref{question-trees-treewidth} and every possible variant of~\Cref{question-parameters} in the negative.
Namely, we prove the following.

\begin{theorem}\label{thm:main-intro}
For every positive integer $t$ and for every $k \in \{2,3,4\}$, there exist connected graphs $H_1,\dots,H_k$ such that
\begin{enumerate}
	\item if $k=4$, then $H_1$, $H_2$, $H_3$, $H_4$ are trees with $\td(H_i) \leq 3$ for each $i \in [4]$;  \label{item:main:k=4}
	\item if $k=3$, then $H_1$, $H_2$ are trees with $\td(H_1),\td(H_2) \leq 3$, and $\pw(H_3) \leq 2$, $\td(H_3) \leq 4$;\label{item:main:k=3}
	\item if $k=2$, then $\pw(H_1) \leq 2$, $\pw(H_2) \leq 3$, and $\td(H_1) \leq 3$, $\td(H_2) \leq 4$; \label{item:main:k=2}
\end{enumerate}


and there exists a graph $G_t$ with $\tw(G_t) \geq t$ such that $\{H_1,\dots,H_k\}$ is an isometric edge-covering of $G_t$.
\end{theorem}

Observe that \cref{thm:main-intro} misses the case of $k=3$ in \cref{question-trees-treewidth} since, in this case $H_3$ is not a tree.
The answer in this case is again negative, which we show with a construction of slightly different flavor.
In particular, the graph that is covered has bounded degree.

\begin{theorem}\label{thm:bd-degree}
For every integer $t \geq 3$, there exists a graph $G_t$ with $\Delta(G_t) \leq 5$ and $\tw(G_t) \geq t$, and there exists a family of trees $\{H_1,H_2,H_3\}$ that is an isometric edge-covering of $G_t$.
\end{theorem}

The graph $G_t$ that we construct in~\cref{thm:main-intro} has treewidth at least $t$ because it contains a subdivision of a wall of order $t$ (the missing definition can be found in \cref{sec:preliminaries}).
In fact, when covering by isometric trees, we can take $G_t$ with even richer structure.

\begin{theorem}\label{thm:main-technical-subcubic}
Let $\Delta \geq 3$ be an integer.
For every connected graph $X$ with $\Delta(X) \leq \Delta$, there exists a graph $G_X$ containing a $5$-subdivision of $X$ as an induced subgraph, and there exists a family of trees $\{H_1,\dots,H_{\Delta+2}\}$, each of radius at most $2$, that is an isometric edge-covering of $G_X$.
Additionally, we may have $|V(G_X)| = O_\Delta(|V(X)|)$.
\end{theorem}

In light of~\Cref{thm:main-technical-subcubic}, any weakening of \cref{question-trees-treewidth} in which $\tw$ is replaced by another parameter $\para{p}$ would require $\para{p}$ to remain bounded on $5$-subdivisions of subcubic graphs.
This is a strong restriction, as this class includes the 5-subdivisions of arbitrarily large constant-degree expanders, on which many standard graph parameters are unbounded.
For example, twin-width is large on most subcubic graphs (and constant subdivisions of them) due to a counting argument~\cite{tww2}. 
Hence, there are graphs of unbounded twin-width edge-covered by five isometric trees.

The \q{additionally} part of \cref{thm:main-technical-subcubic} also implies that \cref{question-trees-treewidth} has a negative answer in a very strong sense.
It is well-known that there exist subcubic graphs with treewidth linear in their number of vertices.
By taking such graphs as $X$, \cref{thm:main-technical-subcubic} implies that there exist graphs edge-coverable by five trees that have treewidth linear in their number of vertices.

Our bounded degree construction in \cref{thm:bd-degree} does not extend easily to a variant as in \cref{thm:main-technical-subcubic}.
Thus, we wondered if bounded degree graphs isometrically edge-coverable by few trees have treewidth sublinear in their number of vertices. Surprisingly, the answer to this question is positive, even in the \q{vertex-cover} setting: being coverable by a fixed number of bounded-degree isometric trees imposes nontrivial structural restrictions on the graph.

\begin{theorem}
    \label{thm:positive_general}
    Let $\Delta$ and $k$ be positive integers.
    For $n$-vertex graphs $G$ with $\Delta(G) \leq \Delta$ admitting isometric vertex-covering with $k$ trees, we have 
    \[\tw(G) =O_{\Delta,k}(n / \log n).\]
\end{theorem}

Our proof of \cref{thm:positive_general} gives a much stronger bound under the assumption that the trees in the covering have \q{polynomial growth}.
Here, for two fixed positive integers $c$ and $d$, we say that a tree $T$ is of \defin{$(c,d)$-growth} when every subtree of $T$ of radius at most $r$ has at most $c\cdot r^d$ vertices for every positive integer $r$.
For example, trees of pathwidth at most $p$ and maximum degree at most $\Delta$, have $((2\Delta)^p,p)$-growth, which can be inferred e.g.\ from a lemma of Suderman \cite[Lemma 7]{S04}.

\begin{theorem}
    \label{thm:positive_growth}
    Let $k$, $c$, and $d$ be positive integers.
    For $n$-vertex graphs $G$ admitting isometric vertex-covering with $k$ trees, each of $(c,d)$-growth, we have 
    \[\tw(G) = O_{k,c,d}(n^{1 - 1/(kd)}).\]
\end{theorem}

\cref{thm:positive_growth} with weaker bounds in the exponent can be inferred from a result of Biswal, Lee, and Rao~\cite[Lemma 5.6]{BLR10} or a result of Krauthgamer and Lee~\cite{KL07}.
We complement \cref{thm:positive_growth} with a qualitatively matching lower bound.

\begin{theorem}
    \label{thm:example-growth}
    Let $m \geq 3$ be an integer.
    There exist $n$-vertex graphs $G$ with $\Delta(G) \leq 3$ admitting isometric vertex-covering with $3^{m-2}$ trees, each of pathwidth at most $2m-2$ such that
    \[\tw(G) = \Omega_m(n^{1 - 2/m}).\]
\end{theorem}

We did not manage to extend the construction witnessing \cref{thm:example-growth} to edge-coverings and we leave this as an open problem.

\Cref{thm:example-growth} shows that bounded degree graphs vertex-coverable by a small number of trees may have very large treewidth. 
Our construction showing this relies on long subdivisions.
We show that these long subdivisions are indeed necessary.

\begin{theorem}
\label{thm:constant-sibdivisions}
    For all positive integers $k$, $d$, $t$, and $s$, there exists a positive integer $g = g(k,d,t,s)$ such that for every graph $G$, if $G$ admits an isometric vertex-covering with $k$ subgraphs, each of maximum degree at most $d$ and treewidth at most $t$, then $G$ does not contain an ($\leq s$)-subdivision of the wall of order $g$ as a subgraph.
\end{theorem}

The plan for the paper is the following.
In \cref{sec:preliminaries}, we introduce definitions and notation.
In \cref{sec:constructions}, we prove the lower-bounds theorems (i.e.  \cref{thm:main-intro,thm:bd-degree,thm:main-technical-subcubic,thm:example-growth}), which require giving a construction.
In \cref{sec:bdd_deg_ub}, we prove positive results for graphs of bounded degree (i.e.\ \cref{thm:positive_general,thm:positive_growth}).
In \cref{sec:constant-subdivisions}, we study constant subdivisions by proving \cref{thm:constant-sibdivisions}.
Finally, we conclude in \cref{sec:open} with some open problems.

\section{Preliminaries}\label{sec:preliminaries}

For a positive integer $n$, write $\defmath{[n]}=\{1,\ldots,n\}$, and let $\defmath{[n]} = \emptyset$ for all nonpositive integers $n$.
We set the maximum over the empty set to be equal to $0$.

For a graph $G$, we denote by \defin{$V(G)$} its vertex set and by \defin{$E(G)$} its edge set.
For a graph $G$, its vertex $v \in V(G)$, and an integer $r$, the \defin{ball of radius $r$ centered at $v$} in $G$ is
\[\defmath{B_G(v,r)} := \{u \in V(G) : \dist_G(v,u) \leq r\}.\]
The \defin{radius} of a graph $G$ is the minimum nonnegative integer $r$ such that there exists $v \in V(G)$ with $V(G) \subset B_G(v,r)$.
For every positive integer $n$, we denote by \defin{$P_n$} the path on $n$ vertices -- we ignore isomorphism issues in this note. 
A \defin{star} is a tree of radius $1$.

Given a graph $G$, a positive integer $c$, and an edge $uv \in E(G)$, the operation of \defin{subdividing} $uv$ in $G$ \defin{$c$ times} returns a graph $G'$ on the vertex set $V(G) \cup\{s_1,\dots,s_c\}$ where $s_1,\dots,s_c \notin V(G)$ and the edge set $(E(G) \setminus \{uv\}) \cup \{us_1,s_cv\} \cup \{s_is_{i+1} : i \in [c-1]\}$.
A graph $H$ is a \defin{subdivision} of a graph $G$ if $H$ can be obtained from $G$ by performing some number of subdivision operations.
For a positive integer $c$, a graph $H$ is a \defin{$c$-subdivision} (resp.\ \defin{($\leq c$)-subdivision}) of a graph $G$ if $H$ is obtained from $G$ by subdividing $c$ times (resp.\ at most $c$ times) each edge of $G$.  

Let $G$ be a graph.
A \defin{tree decomposition} of $G$ is a pair $\mathcal{W} = (T,(W_x : x \in V(T)))$ where $T$ is a tree and $W_x \subseteq V(G)$ for every $x \in V(T)$ satisfying the following conditions:
\begin{enumerate}
    \item for every $u \in V(G)$, $T[\{x \in V(T) : u \in W_x\}]$ is a connected subgraph of $T$, and
    \item for every edge $uv \in E(G)$, there exists $x \in V(T)$ such that $u,v \in W_x$.
\end{enumerate}
The \defin{width} of $\mathcal{W}$ is $\max_{x \in V(T)} |W_x|-1$, and the \defin{treewidth} of $G$ is the minimum width of a tree decomposition of $G$.
A tree decomposition $(T,(W_x : x \in V(T)))$ of $G$ is a \defin{path decomposition} of $G$ if $T$ is a path.
The \defin{pathwidth} of $G$ is the minimum width of a path decomposition of $G$.
The definition of \defin{treedepth} of $G$ is recursive:
\begin{enumerate}
    \item if $G$ has no vertices, then $\td(G) = 0$, 
    \item if $G$ is not connected, then $\td(G) = \max_{C \in \calC} \td(C)$ where $\calC$ is the set of components of~$G$,
    \item if $G$ is connected, then $\td(G) = \min_{v \in V(G)} \td(G-v) + 1$.
\end{enumerate}

For positive integers $n$ and $m$, the \defin{$n \times m$ grid} is the graph on vertex set $[n] \times [m]$ so that a vertex $(a,b)$ is adjacent to another vertex $(c,d)$ if and only if $|a-c| + |b-d| = 1$.
By \defin{square grid of order $t$}, we mean the $t \times t$ grid.
The \defin{wall} of \defin{order} $t$ is the graph obtained from the $t \times (2t+1)$ grid by removing all the edges of the form $(a,b)(a+1,b)$, where $a$ and $b$ have different parity.
See~\Cref{fig:wall-proper}.
It is well-known that for every positive integer $t$, the treewidth of the $t \times t$ grid is exactly $t$.
Moreover, treewidth is monotone under taking minors, and it is easy to check that the $t \times t$ grid is a minor of the wall of order $t$.
Finally, the subdivision operation does not decrease treewidth.
It follows that treewidth of a subdivision of the wall of order $t$ is at least $t$.

\subsection{Retractions and isometries}

Retractions are a very useful tool when working with isometric subgraphs.
In this subsection, we introduce some basic background material on retractions.

Given a graph $G$ and a subgraph $H$ of $G$, we say that $\gamma : V(G) \to V(H)$ is a \defin{retraction} of $G$ onto $H$ (or simply, retraction onto $H$ if $G$ is clear from the context) if for every $u \in V(H)$, $\gamma(u) = u$.
A retraction is called \defin{nonexpansive} if for all $u,v \in V(G)$, we have either $\gamma(u) = \gamma(v)$ or $\gamma(u)\gamma(v) \in E(H)$.
The following observation and its corollary are crucial.

\begin{obs}
    Let $G$ be a graph, $H$ be a subgraph of $G$.
    If there exists a nonexpansive retraction of $G$ onto $H$, then for all $u,v \in V(G)$, we have
    \[\dist_H(\gamma(u),\gamma(v)) \leq \dist_G(u,v).\]
\end{obs}
\begin{proof}
    Let $u,v\in V(G)$, and let $P$ be the shortest path from $u$ to $v$ in $G$.
Each edge of $P$ is mapped either onto an edge or onto a vertex.
Additionally, $\gamma(u),\gamma(v) \in \gamma(P)$.
It follows that $H[\gamma(V(P))]$ contains a path from $\gamma(u)$ to $\gamma(v)$ of length at most the length of $P$, which ends the proof.
\end{proof}

\begin{corollary}\label{cor:retraction-to-isometry}
    Let $G$ be a graph, $H$ be a subgraph of $G$.
    If there exists a nonexpansive retraction of $G$ onto $H$, then $H$ is an isometric subgraph of $G$.
\end{corollary}

We will also use the following easy observation.
\begin{obs}\label{obs:retraction-connected}
    Let $G$ be a graph, let $H$ be a subgraph of $G$, and let $\gamma$ be a nonexpansive retraction of $G$ onto $H$.
    For every connected subgraph $C$ of $G$, $\gamma(V(C))$ induces a connected subgraph of $H$. 
\end{obs}

\begin{proof}
    Suppose that $\gamma(V(C))$ is the union of two subgraphs $A$ and $B$ that are disconnected.
    Since $\gamma$ is nonexpansive, for any $u \in \gamma^{-1}(V(A))$ and $v \in \gamma^{-1}(V(B))$ we have $\dist_G(u,v) \geq \dist_H(\gamma(u),\gamma(v)) \geq 2$, therefore, $G[\gamma^{-1}(V(A))]$ and $G[\gamma^{-1}(V(B))]$ are disconnected in $G$, which is a contradiction.
\end{proof}

\section{Constructions}\label{sec:constructions}

In this section, we prove \cref{thm:main-intro,thm:bd-degree,thm:main-technical-subcubic,thm:example-growth}.

\subsection{Detailed statements}\label{sec:detailed}

The general idea of the constructions in~\Cref{thm:main-intro} is to start with a large subdivided wall (a graph of large treewidth) and to add some vertices and edges so that it is possible to edge-cover the graph by the required isometric subgraphs $H_1,\dots,H_k$.
In fact, the graphs $H_1,\dots,H_k$ that we obtain are even simpler than stated in~\cref{thm:main-intro}.
Below, we give a more precise statement of our result.

For a positive integer $\Delta$, we denote by \defin{$S_\Delta^\star$} the graph obtained by subdividing each edge of the star on $\Delta+1$ vertices.
Given a graph $H$, we denote by \defin{$\Apex(H)$} the class of graphs that contain a vertex whose removal results in a graph each of whose components is a subgraph of $H$.

\begin{theorem}\label{thm:main-technical}
For every positive integer $t$ and for every $k \in \{2,3,4\}$, there exist connected graphs $H_1,\dots,H_k$ such that
\begin{enumerate}
	\item if $k=4$, then $H_1$, $H_2$, $H_3$, $H_4$ are trees of radius at most $2$; \label{item:main:k=4-t}
	\item if $k=3$, then $H_1$, $H_2$ are trees of radius at most $2$, and $H_3 \in \Apex(P_5)$;\label{item:main:k=3-t}
	\item if $k=2$, then $H_1 \in \Apex(P_3)$ and $H_2 \in \Apex(S_3^\star)$; \label{item:main:k=2-t}
\end{enumerate}
and there exists a graph $G_t$ containing a subdivision of a wall of order $t$ as an induced subgraph such that $\{H_1,\dots,H_k\}$ is an isometric edge-covering of $G_t$.
\end{theorem}

It is easy to verify that~\Cref{thm:main-technical} implies~\Cref{thm:main-intro}.
Similarly as~\cref{thm:main-technical-subcubic} is a strengthening of~\Cref{thm:main-intro} under condition~\ref{item:main:k=4} (covering by isometric trees), we obtain a strengthening of~\Cref{thm:main-intro} under condition~\ref{item:main:k=2}.

\begin{theorem}\label{thm:main-technical-any}
Let $\Delta \geq 3$ be an integer.
For every connected graph $X$ with $\Delta(X) \leq \Delta$, there exist $H_1 \in \Apex(P_3)$ and $H_2 \in \Apex(S_\Delta^\star)$ (in particular, $\td(H_1) \leq 3$ and $\td(H_2) \leq 4$), and there exists a graph $G_X$ containing a subdivision of $X$ as an induced subgraph such that $\{H_1,H_2\}$ is an isometric edge-covering of $G_X$.
\end{theorem}

\subsection{Covering by four isometric trees}\label{sec:k=4}
In this subsection, we prove~\cref{thm:main-technical} under condition~\ref{item:main:k=4} and we prove~\cref{thm:main-technical-subcubic}.

Let $X$ be a graph.
The \defin{set of incidences} of $X$ is the set
\[\mathrm{\defin{$\Inc(X)$}} := \{(u,e) : u \in V(X), e \in E(X), u \in e\}.\]
Let $\kappa$ be a positive integer.
We say that a function $\varphi : \Inc(X) \rightarrow [\kappa]$ is \defin{proper} if for all $uv \in E(X)$, we have $\varphi(u,uv) \neq \varphi(v,uv)$, and for all $u \in V(X)$ if $uv$ and $uw$ are distinct edges of $X$, then $\varphi(u,uv) \neq \varphi(u,uw)$.

It is clear that by a greedy procedure, for every graph $X$ with maximum degree $\Delta$, we can find a proper function $\varphi : \Inc(X) \rightarrow [\Delta+1]$. Indeed, each element $(u,uv) \in \Inc(X)$ has at most $\Delta-1$ color constraints from the set $\{(u,e): e \in E(X)\setminus \{uv\}, u \in e\}$, and at most one constraint from $(v,uv)$, therefore $\Delta-1+1+1 = \Delta+1$ colors are always enough.
A wall $X$ has maximum degree at most $3$, and thus admits a proper function $\varphi : \Inc(X) \rightarrow [4]$. 
In fact, the symmetric structure of walls implies that every wall $X$ admits a proper function $\varphi : \Inc(X) \rightarrow [3]$: one can just extend the construction shown in~\Cref{fig:wall-proper}.
These observations and the next lemma yield the advertised statements.

\begin{figure}[!tp]
  \centering
  \includegraphics{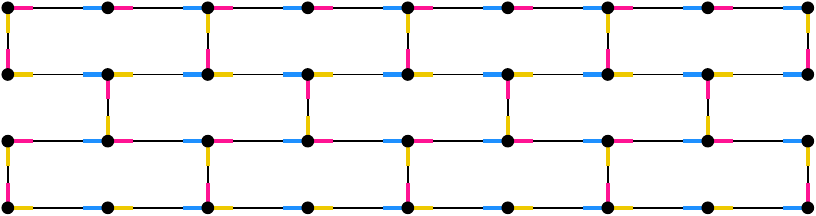}
  \caption{
        A wall $X$ of order $4$ along with a proper function mapping $\Inc(X)$ to three colors: blue, pink, and yellow.
    }
   \label{fig:wall-proper}
\end{figure}

\begin{lemma}\label{lem:proper-to-result}
	Let $X$ be a connected graph and let $\kappa$ be a positive integer.
	If there exists a proper $\varphi : \Inc(X) \rightarrow [\kappa]$, then there exists a graph $G$ that contains a subdivision of $X$ as an induced subgraph and there exists a family $\{H_1,\dots,H_{\kappa+1}\}$ of trees of radius at most $2$ that is an isometric edge-covering of $G$.
\end{lemma}

\begin{figure}[!tp]
  \centering
  \includegraphics{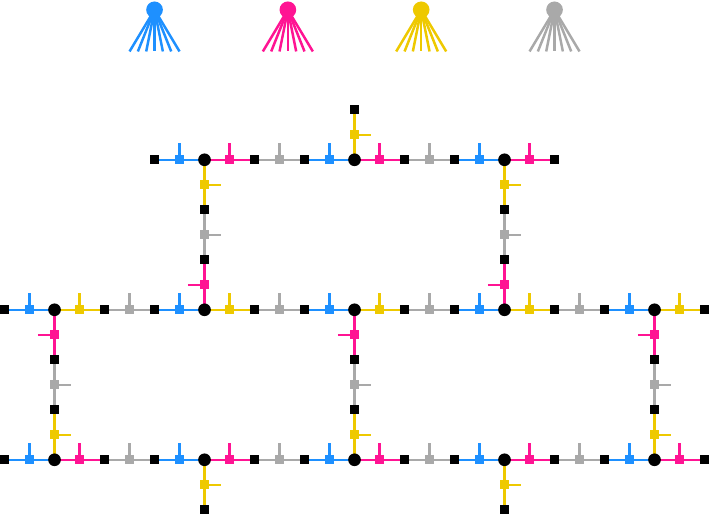}
  \caption{
        An illustration of the construction in the proof of~\Cref{lem:proper-to-result}.
        Vertices that are results of subdividing edges of the original wall are marked as squares.
        Vertices belonging to exactly one subgraph $H_i$ are colored appropriately.
        Note that we used the function $\varphi$ as in~\Cref{fig:wall-proper}.
        Vertices $A = \{a_1,a_2,a_3,a_4\}$ are on the top of the figure.
        For clarity, we do not draw edges incident to vertices in $A$, instead, we only draw their beginnings and ends.
    }
   \label{fig:wall}
\end{figure}

\begin{proof}
	See~\Cref{fig:wall}.
	Let $\varphi : \Inc(X) \rightarrow [\kappa]$ be proper.
	Let $G'$ be obtained from $X$ by subdividing each edge five times and adding isolated vertices $a_1,\dots,a_{\kappa+1}$.
	We will construct $G$ by adding some edges to $G'$ as described below.
	In parallel, we will construct the graphs $H_1,\dots,H_{\kappa+1}$. 
    Note that these graphs will be induced subgraphs of $G$, hence, we will only specify $V(H_i)$ for each $i \in [\kappa+1]$.
	We initiate by adding $a_i$ in $V(H_i)$ for every $i \in [\kappa+1]$.
	The next step is performed for every edge $uv \in E(X)$ independently.
	Let $uv \in E(X)$ and let $s_1,s_2,s_3,s_4,s_5 \in V(G') \setminus V(X)$ be the unique sequence of vertices such that $us_1s_2s_3s_4s_5v$ is a path in $G'$. 
	First, we add edges $a_{\varphi(u,uv)}s_1$, $a_{\kappa+1}s_3$, and $a_{\varphi(v,uv)}s_5$ to $G$. Next, we cover $u,v,s_1,\ldots,s_5$ by $H_1,\ldots,H_{\kappa+1}$:
    \begin{itemize}
        \item Add vertices $u,s_1,s_2$ to $H_{\varphi(u,uv)}$.
        \item Add vertices $s_2,s_3,s_4$ to $H_{\kappa+1}$.
        \item Add vertices $s_4,s_5,v$ to $H_{\varphi(v,uv)}$.
    \end{itemize}
	This completes the construction of $G$ and $H_1,\dots,H_{\kappa+1}$.
	
	By construction, $G$ contains a subdivision of $X$ as an induced subgraph.
	  It is also direct that $\{H_1,\dots,H_{\kappa+1}\}$ edge-cover $G$ and that $H_i$ is a tree of radius at most $2$ for each $i \in [\kappa+1]$.
	To complete the proof, it suffices to verify that $H_i$ is an isometric subgraph of $G$ for every $i \in [\kappa+1]$.
	Let $i \in [\kappa+1]$ and let $x$ and $y$ be two distinct vertices of $H_i$.
    Since $H_i$ is a tree of radius at most $2$, we have $\dist_{H_i}(x,y) \leq 4$. 
    If $x$ and $y$ lie on the same subdivided edge of $G$ then $\dist_{H_i}(x,y) \leq 2$ and $\dist_G(x,y) = \dist_{H_i}(x,y)$. 
    Otherwise, $x$ and $y$ lie on different subdivided edges, and the fact that $\varphi$ is proper ensures $\dist_{G'}(x,y) \geq 4$, going over all paths starting from $x$ of length at most $4$ in $G$ it is easy to verify that $\dist_G(x,y) = \dist_{H_i}(x,y) \leq 4$.
	This completes the proof of the lemma.
\end{proof}

Note that~\cref{lem:proper-to-result} applied to a wall $X$ with a proper function $\varphi : \Inc(X) \rightarrow [3]$ implies \cref{thm:main-technical} under condition~\ref{item:main:k=4-t} (and therefore \cref{thm:main-intro} under condition~\ref{item:main:k=4}). Applying \cref{lem:proper-to-result} to any graph $X$ of maximum degree $\Delta$ and any greedy proper function $\varphi : \Inc(X) \to [\Delta+1]$ implies \cref{thm:main-technical-subcubic}.

\subsection{Covering by three isometric simple subgraphs}\label{sec:k=3}

Next, we prove~\cref{thm:main-technical} with~\ref{item:main:k=3-t} assumed.
The construction is similar to the one given in~\Cref{sec:k=4}.
Again, we start with a wall $X$, subdivide each edge several times (\q{horizontal} edges five times and \q{vertical} edges seven times), and add three isolated vertices.
Next, we add some edges and distribute them into the subgraphs $H_1$, $H_2$, and $H_3$.
The construction is presented in~\Cref{fig:wall-3-colors}.
It is easy to extend the construction to a wall of any order, and to verify that it satisfies the required properties (similarly as in~\Cref{sec:k=4}).

\begin{figure}[!tp]
  \centering
  \includegraphics{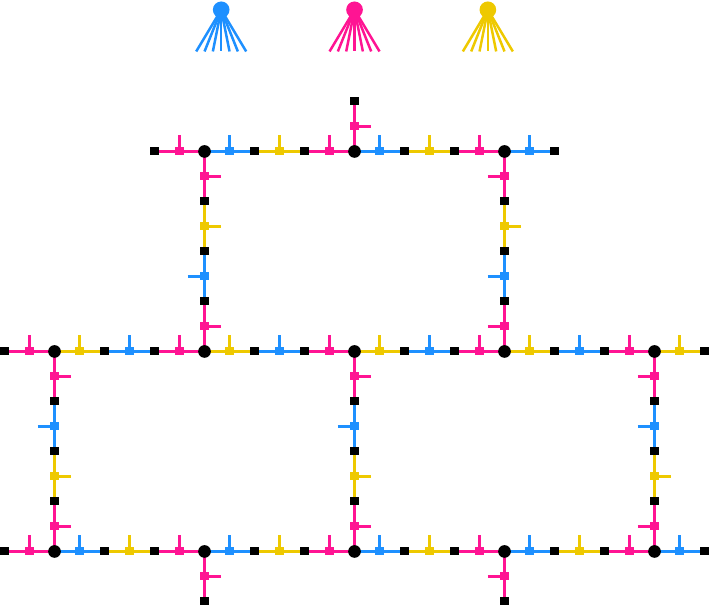}
  \caption{
        We use the same conventions as in~\Cref{fig:wall}.
        Note that the graphs $H_i$ corresponding to colors blue and yellow are trees of radius at most $2$, and the graph corresponding to color pink is in $\Apex(P_5)$.
    }
   \label{fig:wall-3-colors}
\end{figure}

\subsection{Covering by two simple graphs}\label{sec:k=2}
In this subsection, we prove~\cref{thm:main-technical} under condition~\ref{item:main:k=2-t} and we prove~\cref{thm:main-technical-any}.
Note that since every wall is subcubic, the latter implies the former.
The key is the following straightforward observation.

Given a graph $G'$ and $P=(V_1,V_2)$ with $V_1,V_2 \subseteq V(G')$, let $G$ be constructed from $G'$ by adding two new vertices $a_1$ and $a_2$ so that $a_i$ is adjacent to all vertices in $V_i$ for each $i \in [2]$. We also define $H_i = G[V_i \cup \{a_i\}]$ for each $i \in [2]$. We say that $(G,H_1,H_2)$ is \defin{built} from $(G',P)$.

\begin{obs}\label{obs:isometric-subgraphs}
Let $G'$ be a graph and let $P = (V_1,V_2)$. Let $(G,H_1,H_2)$ be built from $(G',P)$.
The graphs $H_1$ and $H_2$ are isometric subgraphs of $G$.
\end{obs}

\begin{proof}
    Let $i \in [2]$ and let $x,y \in V(H_i)$ with $x \neq y$.
    From the construction, $H_i$ has radius $1$, and so, $\dist_{H_i}(x,y) \in \{1,2\}$.
    Since $H_i$ is an induced subgraph of $G$, if $x$ and $y$ are nonadjacent in $H_i$, then they are nonadjacent in $G$.
    This completes the proof of the observation.
\end{proof}

\begin{proof}[Proof of~\Cref{thm:main-technical-any}]
	The construction is very similar to the one in~\Cref{lem:proper-to-result}.
	Let $G'$ be obtained from $X$ by subdividing each edge five times.
	We construct $P = (V_1,V_2)$.
	The next step is performed for every edge $uv \in E(X)$ independently.
	Let $uv \in E(X)$ and let $s_1,s_2,s_3,s_4,s_5 \in V(G') \setminus V(X)$ be such that $us_1s_2s_3s_4s_5v$ is a path in $G'$.
	We set $\{u,s_1,s_2,s_4,s_5,v\} \subseteq V_2$ and $\{s_2,s_3,s_4\} \subseteq V_1$. 
	Let $(G,H_1,H_2)$ be built from $(G',P)$.
    \Cref{obs:isometric-subgraphs} implies that $H_1$ and $H_2$ are isometric subgraphs of $G$.
	By construction, $\{H_1,H_2\}$ edge-cover $G$.
	Also by construction, $G$ contains a subdivision of $X$ as an induced subgraph.
	Finally, note that every component of $G[V_1]$ is a path on three vertices and every component of $G[V_2]$ is an $S_d^\star$ for some positive integer $d$ at most the maximum degree of $X$.
	Since $H_i = G[V_i \cup \{a_i\}]$ for each $i \in [2]$, it follows that $H_1 \in \Apex(P_3)$ and $H_2 \in \Apex(S_\Delta^\star)$, as desired.
\end{proof}

\subsection{Covering bounded degree graphs by three isometric trees}\label{sec:bdd_deg_lb}

    In this subsection, we prove \cref{thm:bd-degree}.

\begin{figure}[!tp]
    \centering
    \includegraphics[width=\linewidth]{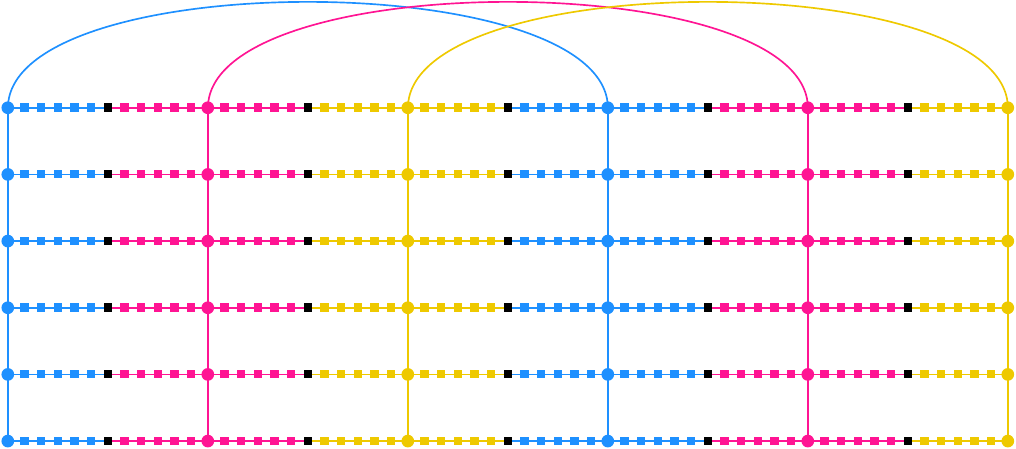}
    \caption{An illustration of the construction witnessing \cref{thm:bd-degree} for $t = 6$.
    The conventions again match \cref{fig:wall}.}
    \label{fig:3treesGrid}
\end{figure}

\begin{proof}[Proof of \cref{thm:bd-degree}] 
Let $t \geq 3$ be an integer. 
See \Cref{fig:3treesGrid}. 
We start the construction of $G_t$ by considering a $t \times t$ grid $X$. 
Each horizontal edge of the grid, i.e.\ an edge $(a,b)(a+1,b)$, is next subdivided $2t-1$ times, we denote by $P_{a,b}$ the obtained path between $(a,b)$ and $(a+1,b)$. 
Also, we denote by $P_{a,b}^-$, the subpath induced by the $t+1$ closest vertices to $(a,b)$ of $P_{a,b}$, and by $P_{a,b}^+$ the subpath induced by the $t+1$ closest vertices to $(a+1,b)$ of $P_{a,b}$. 
Next, for every $a \in [t-3]$, we add an edge between $(a,1)$ and $(a+3,1)$. 
The resulting graph is $G_t$. Note that $G_t$ has maximum degree $5$ and it contains a subdivision of the $t \times t$ grid as a subgraph. 
To complete the proof, it suffices to construct the three isometric trees edge-covering $G_t$. 
For every $a \in [t]$, let $C_a$ be the $a$th column of $G_t$, i.e.\ $C_a := \{(a,b) : b \in [t]\}$. 
For each $i \in [3]$, we define 
\[T_i = G_t\left[\bigcup_{\substack{a \in [t] \\ a \equiv i \bmod 3}} V(C_a) \cup \bigcup_{\substack{a \in [t-1] \\ a \equiv i \bmod 3}} \bigcup_{b \in [t]} V(P_{a,b}^-) \cup \bigcup_{\substack{a \in [t]\setminus\{1\} \\ a \equiv i \bmod 3}} \bigcup_{b \in [t]} V(P_{a-1,b}^+) \right].\] 
By construction, $\{T_1,T_2,T_3\}$ edge-covers $G_t$. 
For every $i \in [3]$, $T_i$ is a tree, and it is easy to check that $T_i$ is an isometric subgraph of $G_t$. 
\end{proof}

\subsection{Covering bounded degree graphs with isometric trees of bounded pathwidth}
Finally, in this subsection, we prove \cref{thm:example-growth}.

A \defin{configuration} is a triple $(G,\mathcal T,C)$, where $G$ is a graph, $\mathcal T$ is an isometric vertex-covering of $G$ where each graph in $\calT$ is a rooted tree, and $C\subseteq V(G)$. 
For each $T \in \calT$, we denote by \defin{$r(T)$}, its \defin{root}.
The vertices in $C$ are called \defin{centers} of the configuration.
The configuration is \defin{extendable} if, for every $T\in\mathcal T$ and every $v\in V(T)$,
\begin{equation*}
  \dist_G(v,r(T)) = \min_{T'\in\mathcal T}\dist_G(v,r(T')).
\end{equation*}
Note that if a vertex belongs to several members of $\calT$, then all the roots of all those members are tied for being closest to that vertex.

A \defin{leaf} of a graph is a vertex of degree $1$.
For every $c\in C$, fix a tree \defin{$T(c)$} in $\calT$ such that $c$ is its vertex.
The configuration $(G,\calT,C)$ is \defin{clean} if
\begin{enumerate}
    \item no two trees in $\calT$ share a root and no root is a center,
    \item for every $v \in C \cup \{r(T) : T \in \calT\}$, $v$ is a leaf of $G$ and $v$ is a vertex of only one tree in $\calT$.
\end{enumerate}

We will now show how to \q{extend} a clean extendable configuration.
To this end, for each such configuration $(G,\calT,C)$ and for each positive integer $p$, we define the configuration \defin{$\calE_p(G,\calT,C)$}.
Along the way, we introduce the notation that will be used throughout when proving the properties of configurations.

Let $(G,\calT,C)$ be a clean extendable configuration and let $p$ be a positive integer.  
First, we construct the graph $G_\calE$.
Let
\begin{equation*}
  \ell:=\max\left\{2,\max_{c\in C} \dist_G(c,r(T(c)))\right\}.
\end{equation*} 
Take $3p$ pairwise disjoint copies $G_1,\ldots,G_{3p}$ of $G$. 
For every $i \in [3p]$ and every $v\in V(G)$, we write $v_i$ for the copy of $v$ in $G_i$.
We use the same convention for subgraphs and vertex sets.  
For every $c\in C$ and $i\in[3p-1]$, we add a path \[Q_i(c):=c_i^0c_i^1\dots c_i^{2\ell}c_i^{2\ell+1}\] where $c_i^0=c_i$ and $c_i^{2\ell+1}=c_{i+1}$, and whose $2\ell$ internal vertices are new.  
The union of these paths for a fixed $c \in C$ is a path, denoted by $L_c$, containing $c_1,\ldots,c_{3p}$ in this order.  
For every $i\in[3p]$, let \[S_i(c):=B_{L_c}(c_i,\ell).\]
Since for each $i \in [3p-1]$, vertices $c_i$ and $c_{i+1}$ are at distance $2\ell+1$ in
$L_c$, the sets $S_i(c)$ for $i \in [3p]$ form a partition of $V(L_c)$. 
Also, for every $T\in\calT$ and every $m\in[3]$, add the path
\[
  P_{T,m}:=r(T)_m r(T)_{m+3}\cdots r(T)_{m+3(p-1)}.
\]
For every $T\in\calT$ and every $m\in[3]$, we add a new leaf $\rho_{T,m}$ adjacent to $r(T)_m$.
For every $c\in C$ and $i\in [3p]$, let $z_i(c)$ be the vertex of $S_i(c)$ at distance $\ell$ from $c_i$ in the direction of $c_{i+1}$ if $i<3p$, and in the direction of $c_{i-1}$ if $i=3p$.  
Add a new leaf $\overline{c_i}$ adjacent to $z_i(c)$.  
This completes the construction of $G_\calE$.

Next, we construct the family $\calT_\calE$.
All $T \in \calT$ and $m \in [3]$ will produce a new tree in the following way:
\[T_{\calE,m} := \bigcup_{j=0}^{p-1} \left(T_{m+3j} \cup \bigcup_{\substack{c \in C \\ T(c) = T}} G_\calE[S_{m+3j}(c) \cup \{\overline{c_{m+3j}}\}] \right) \cup P_{T,m} \cup G_\calE[\{r(T)_m,\rho_{T,m}\}].\]
Note that $T_{\calE,m}$ is indeed a tree.
We set $r(T_{\calE,m}) = \rho_{T,m}$.
Let
\[
  \calT_\calE:=\{T_{\calE,m}:T\in\calT,\ m\in[3]\} \text{ and }
  C_\calE:=\{\overline{c_i}:c\in C,\ i \in [3p]\}.
\]
Finally, for every $c \in C$ and $i \in [3p]$, we set $T(\overline{c_i}) = T_{\calE,m}$ where $T(c) = T$ and $i \equiv m \bmod 3$.

\begin{lemma}\label{lem:p-extension}
Let $(G,\calT,C)$ be a clean extendable configuration and let $p$ be a positive integer.
Then $\mathcal E_p(G,\calT,C)=(G_\calE,\calT_\calE,C_\calE)$ is again a clean extendable configuration, and, moreover,
\begin{enumerateAlpha}
    \item $|V(G_\calE)|=3p|V(G)|+2\ell(3p-1)|C|+3p|C|+3|\calT|$, \label{item:p-extension:G}
    \item $|\calT_\calE| = 3 |\calT|$, \label{item:p-extension:T}
    \item $|C_\calE|=3p|C|$, \label{item:p-extension:C}
    \item $\Delta(G_\calE)\leq\max\{\Delta(G),3\}$, \label{item:p-extension:Delta}
    \item $\pw(T_{\calE,m}) \leq \pw(T) + 2$ for every $T \in \calT$ and $m \in [3]$, \label{item:p-extension:pw}
    \item with $\ell_\calE:=\max\left\{2,\max_{c\in C_\calE} \dist_{G_\calE}(c,r(T(c)))\right\}$, we have $\ell_\calE \leq 2\ell+p+1$. \label{item:p-extension:ell}
\end{enumerateAlpha}
\end{lemma}
\begin{proof}
    Directly from the construction, we obtain that $\calT_{\calE}$ is a collection of trees that vertex-cover $G_\calE$.
    It is also immediate to check \ref{item:p-extension:G}, \ref{item:p-extension:T}, \ref{item:p-extension:C}, \ref{item:p-extension:Delta}, and \ref{item:p-extension:ell}.
    Additionally, if $\calE_p(G,\calT,C)$ is a configuration, then it is a clean configuration again by construction.

    We argue that \ref{item:p-extension:pw} holds.
    Let $T \in \calT$ and $m \in [3]$.
    The tree $T_{\calE,m}$ is obtained from disjoint copies of $T$ by connecting the roots in the copies into a path and adding hanging paths to some vertices. 
    It is easy to see that each of these operations increases pathwidth by at most one.

    It remains to verify that every tree in $\calT_\calE$ is an isometric subgraph of $G_\calE$ and that $\calE_p(G,\calT,C)$ is extendable.

First, we prove that every tree in $\calT_\calE$ is an isometric subgraph of $G_\calE$.
Adding leaves to the ambient graph cannot shorten a path between old vertices, and a distance involving a newly added leaf is obtained by adding one to the distance from its neighbor.
Therefore, we may remove the leaves of $G_\calE$ and it suffices to prove that the trees obtained from the trees of $\calT_\calE$ by removing the same leaves of $G_\calE$ are isometric in $G_\calE$.
Let $G_\calE^\circ$ be the graph obtained from $G_\calE$ by removing all the vertices of the form $\rho_{T,m}$ and $\overline{c_i}$ for all $T \in \calT$, $m \in [3]$, $c \in C$, and $i \in [3p]$.
Next, let $T_{\calE,m}^\circ$ be obtained from $T_{\calE,m}$ by removing the same vertices.
It follows that the following statement suffices to show that every tree in $\calT_\calE$ is an isometric subgraph of $G_\calE$.

\begin{claim}\label{claim:isometric}
    For all $T \in \calT$ and $m \in [3]$, $T_{\calE,m}^\circ$ is an isometric subgraph of $G_{\calE}$.
\end{claim}
\begin{proofclaim}
    Let $T \in \calT$ and $m \in [3]$.
    Let
        \[U_m:=\bigcup_{j=0}^{p-1}\left(V(G_{m+3j})\cup\bigcup_{c\in C}S_{m+3j}(c)\right).\]
    We prove the claim in two steps.
    First, we prove that
    \begin{align} \label{eq:U-in-G}
        \text{$G_\calE^\circ[U_m]$ is isometric in $G_\calE^\circ$.}
    \end{align}
    And then,
    \begin{align}\label{eq:T-in-U}
        \text{$T_{\calE,m}^\circ$ is isometric in $G_\calE^\circ[U_m]$.}
    \end{align}
    Clearly, \eqref{eq:U-in-G} and \eqref{eq:T-in-U} imply the required assertion.
    We first prove \eqref{eq:U-in-G}.
    We construct a nonexpansive retraction $\gamma : V(G_\calE^\circ) \to U_m$, which suffices by \cref{cor:retraction-to-isometry}.
    
    For each $i\in[3p]$, let $\lambda_m(i)$ be the closest element of the set $\{m,m+3,\ldots,m+3(p-1)\}$ to $i$.
    For every $v \in V(G)$ and $i \in [3p]$, we set $\gamma(v_i) = v_{\lambda_m(i)}$.
    It remains to define $\gamma$ on the vertices of each $Q_i(c)$ for $i \in [3p-1]$ and $c \in C$. 
    We map the ordered vertices of $Q_i(c)$ to a lazy walk of length $2\ell+1$ in $G_\calE^\circ[U_m]$ as follows.
    If $\lambda_m(i)=\lambda_m(i+1)=i$, use the lazy walk $c_i^0 \dots c_i^\ell c_i^\ell \dots c_i^0$.
    If $\lambda_m(i)=\lambda_m(i+1)=i+1$, use the symmetric lazy walk $c_i^{2\ell+1}\dots c_i^{\ell+1}c_i^{\ell+1}\dots c_i^{2\ell+1}$.
    If $\lambda_m(i)=\lambda_m(i+1)=a$ and $a \notin \{i,i+1\}$, then we map all the vertices to $c_{\lambda_m(i)}$.
    Finally, assume that $\lambda_m(i)=a$ and $\lambda_m(i+1)=a+3$.
    We use the vertices of the path from $c_a$ to $r(T(c))_a$ in $T(c)_a$ and the vertices of the path from $r(T(c))_{a+3}$ to $c_{a+3}$ in $T(c)_{a+3}$ in this order so that the edge $c_i^\ell c_i^{\ell+1}$ is mapped to the edge $r(T(c))_ar(T(c))_{a+3}$.
    Note that for this to happen, we may have to repeat some vertices.
    However, $\dist_G(c,r(T(c))) \leq \ell$ for every $c \in C$ by the definition, so such a mapping exists.
    It is not hard to verify that $\gamma$ is a nonexpansive retraction of $G_\calE^\circ$ onto $G_\calE^\circ[U_m]$, as desired.
    This completes the proof of \eqref{eq:U-in-G}.

    Next, we prove \eqref{eq:T-in-U}.
    For every integer $j$ with $0 \leq j < p$, let
    \[G_j^* := G_\calE^\circ\left[V(G_{m+3j}) \cup \bigcup_{c \in C}S_{m+3j}(c).\right]\]
    Next, consider $x$ and $y$ in $V(G_j^*) \cap V(T_{\calE,m})$. 
    It is easy to see that there is a nonexpansive retraction of $G_\calE^\circ[U_m]$ onto $G_{m+3j}$.
    For every $i \in [3p]$ with $i \equiv m \bmod 3$, it suffices to map the vertices of $G_{i}$ onto $G_{m+3j}$ labelwise and all the vertices of each $S_{i}(c)$ to $c_{m+3j}$ for every $c \in C$.
    Thus, $G_{m+3j}$ is an isometric subgraph of $G_\calE^\circ[U_m]$.
    Since $T$ is isometric in $G$, and every $S_{m+3j}$ for an integer $j$ with $0 \leq j < p$ is attached only at its center, the unique path in $T_{\calE,m}^\circ$ between $x$ and $y$ is shortest.
    
    Now let $x$ lie in the block $G_a^*$ and let $y$ lie in the block $G_b^*$ for some integers $a$ and $b$ with  $0 \leq a < b < p$.
    Every path from $x$ to $y$ in $G_\calE^\circ[U_m]$ uses at least $b-a$ edges of $P^* := \bigcup_{T \in \calT, m \in [3]} P_{T,m}$ since there are no other edges between $G_j^*$ and $G_{j'}^*$ for distinct $j$ and $j'$.
    Let $P$ be a shortest path from $x$ to $y$ in $G_{\calE}^\circ[U_m]$.
    Before using its first edge of $P^*$, $P$ must reach some root in $G_a^*$.
    Similarly, after using its last edge of $P^*$, $P$ leaves a root in $G_b^*$.
    By extendability, $r(T)$ is a closest root to every vertex of $T$.  
    It follows that
    \[
        \dist_{G_{\calE}^\circ[U_m]}(x,y) = |E(P)| \geq \dist_{T_{\calE,m}^\circ}(x,r(T)_{m+3a})+(b-a)+\dist_{T_{\calE,m}^\circ}(r(T)_{m+3b},y).
    \]
    It is clear from the construction that there is a path from $x$ to $y$ of this length in $T_{\calE,m}^\circ$.
    This completes the proof of \eqref{eq:T-in-U} and the claim.
\end{proofclaim}
    
    We now know that $\calE_p(G,\calT,C)$ is a configuration and the next step is to prove that it is extendable.
    For every $v \in V(G)$, let
\[
  h(v):=\min_{T \in \calT}\dist_G(v,r(T)).
\]
Next, we define a function $H$ on $V(G_\calE)$ as follows.  
\[
H(x) := \begin{cases}
    1+\left\lfloor\frac{i-1}{3}\right\rfloor+h(v) & \text{if } x = v_i \text{ for some } v \in V(G) \text { and } i \in [3p],\\[0.4em]
    1+\left\lfloor\frac{i-1}{3}\right\rfloor+h(c)
        +\dist_{L_c}(x,c_i) & \text{if } x \in S_i(c) \text{ for some } c \in C \text{ and } i \in [3p],\\
        H(z_i(c))+1 & \text{if } x = \bar{c_i} \text{ for some } c \in C \text{ and } i \in [3p],\\
        0 & \text{if } x = \rho_{T,m}\text{ for some } T \in \calT \text{ and } m \in [3].
\end{cases}
\]
By the definition and the construction it is clear that for every edge $uv$ of $G_\calE$, we have $|H(u) - H(v)| \leq 1$.
Since on the roots of the trees in $\calT_\calE$, the value of $H$ is $0$, we obtain $H(x) \leq \min_{T \in \calT, m \in [3]} \dist_{G_\calE}(x,r(T_{\calE,m}))$ for every $x \in V(G_\calE)$.
On the other hand, if $x\in V(T_{\calE,m})$, then the unique path in $T_{\calE,m}$ from $x$ to its root $\rho_{T,m}$ has length exactly $H(x)$.  
Thus the new configuration is extendable.
This completes the proof of the lemma.    
\end{proof}

For positive integers $d$ and $p$, let
\[
  \defmath{Q_d(p)}:=\underbrace{P_p\mathbin{\square}\cdots\mathbin{\square}P_p}
                    _{d\text{ factors}}.
\]
be the $d$-dimensional grid of side length $p$.  
A minor model $(B_x:x\in V(H))$ of a graph $H$ in $G$ is called \defin{$C$-centered} for a set $C \subset V(G)$ if every branch set $B_x$ for $x \in V(H)$ contains exactly one vertex of $C$.
The following statement directly implies \cref{thm:example-growth}.

\begin{theorem}\label{thm:iterated-extension}
For every integer $m\geq 2$ there are constants $a_m,b_m>0$ such that for every integer $p\geq 2$, there exists a clean extendable configuration $(G_m,\calT_m,C_m)$ with the following properties:
\begin{enumerate}
  \item $|V(G_m)|\leq b_mp^m$;\label{item:thm:V}
  \item $|\calT_m|=3^{m-2}$;\label{item:thm:T}
  \item $|C_m|=3^{m-2}p^{m-1}$;\label{item:thm:C}
  \item $\Delta(G_m)\leq 3$;\label{item:thm:Delta}
  \item $\pw(T) \leq 2(m-1)$ for every $T \in \calT_m$;\label{item:thm:pw}
  \item with $\ell_m:=\max\{2,\max_{c\in C_m} \dist_{G_m}(c,r(T(c)))\}$, we have $\ell_m\leq a_mp$;\label{item:thm:ell}
  \item $G_m$ contains a $C_m$-centered model of $Q_{m-1}(p)$.\label{item:thm:Q}
\end{enumerate}
Consequently, $\tw(G_m)=\Omega_m(p^{m-2}) = \Omega_m(|V(G_m)|^{1 -2/m})$.
\end{theorem}

\begin{proof}
We proceed by induction on $m$.
For $m=2$, let $x_1\dots x_p$ be a path.
Attach a leaf $c_j$ to $x_j$ for every $j\in[p]$, and attach another leaf $r$ to $x_1$.
Let $G_2$ be the resulting tree, rooted at $r$, let $\calT_2:=\{G_2\}$, and let
$C_2:=\{c_1,\dots,c_p\}$, with $T(c)=G_2$ for every $c\in C_2$.
This is a clean extendable configuration, and it is immediate to verify
\ref{item:thm:V} to \ref{item:thm:Q}.

Assume that the theorem holds for an integer $m\geq 2$, and define
\[
  (G_{m+1},\calT_{m+1},C_{m+1})
  :=\mathcal E_p(G_m,\calT_m,C_m).
\]
\Cref{lem:p-extension} immediately gives \ref{item:thm:V} to \ref{item:thm:ell}.
It remains to prove \ref{item:thm:Q}.

Let $(B_x:x\in V(Q_{m-1}(p)))$ be a $C_m$-centered model in $G_m$, and let
$c_x$ be the unique center in $B_x$ for each $x\in V(Q_{m-1}(p))$.
In the following part, all the \q{configuration} notation is used for
$(G_m,\calT_m,C_m)$.

For every $j\in[p]$ and $c\in C_m$, let
\[
  A_j(c):=\bigcup_{i\in[3p],\,|i-3j|\leq 1}S_i(c).
\]
Let $c\in C_m$.
For each $j\in[p]$, the path $A_j(c)$ contains $c_{3j}$ and $z_{3j}(c)$;
the paths $A_1(c),\dots,A_p(c)$ are pairwise disjoint; and, for
$j\in[p-1]$, there is an edge of $L_c$ between $A_j(c)$ and $A_{j+1}(c)$.

For clarity, let $\overline{c_{x,i}}:=\overline{(c_x)_i}$ for every $x\in V(Q_{m-1}(p))$ and $i\in[3p]$.
For all $x\in V(Q_{m-1}(p))$ and $j\in[p]$, define
\[
  B'_{(x,j)}
  :=(B_x)_{3j}\cup A_j(c_x)\cup\{\overline{c_{x,3j}}\}.
\]
We claim that $(B'_{(x,j)}:x\in V(Q_{m-1}(p)),\,j\in[p])$ is a $C_{m+1}$-centered model of $Q_m(p)$ in $G_{m+1}$.

Each $B'_{(x,j)}$ is connected because $(B_x)_{3j}$ meets $A_j(c_x)$
exactly in $(c_x)_{3j}$ and $\overline{c_{x,3j}}$ is adjacent to
$z_{3j}(c_x)$.
Adjacencies in the first $m-1$ coordinates are inherited from the copy
$G_{3j}$, while adjacency in the last coordinate is supplied by the edge
of $L_{c_x}$ between $A_j(c_x)$ and $A_{j+1}(c_x)$.
Moreover, $B'_{(x,j)}$ contains the unique new center
$\overline{c_{x,3j}}$.
This completes the induction.

The \q{consequently} part of the theorem follows e.g.\ from a result of
Hickingbotham and Wood \cite[Corollary~26]{HickingbothamWood2025}.
\end{proof}

\section{Treewidth of bounded degree graphs coverable by few trees}\label{sec:bdd_deg_ub}


In this section, we prove \cref{thm:positive_general,thm:positive_growth}.
%
%
%
%
The following lemma is well-known in metric graph theory (see, e.g.\ a paper by Hell and Rival~\cite[corollary to Theorem 1]{hell1987absolute}). We add a proof for completeness.

\begin{lemma}[Folklore]
    \label{lem:retraction}
    Let $G$ be a graph, and let $T$ be a tree that is an isometric subgraph of $G$.
    There exists a nonexpansive retraction of $G$ onto $T$.
\end{lemma}

\begin{proof}
For a set $X$ with $V(T) \subseteq X \subseteq V(G)$, we say that a map $\gamma_X : X \to V(T)$ is a nonexpansive partial retraction of $X$ onto $T$ whenever for every $u \in V(T)$, $\gamma_X(u) = u$, and $\dist_G(\gamma_X(u), \gamma_X(v)) \leqslant \dist_G(u, v)$.
Let $\gamma_{V(T)}$ be the identity map on $V(T)$, note that it is a nonexpansive partial retraction of $V(T)$ onto $T$.

Let $V(T) \subseteq X \subseteq V(G)$ be an inclusion-wise maximal set that admits a nonexpansive partial retraction $\gamma_X$ onto $T$.
If $X = V(G)$, then the lemma is proved; so we assume for the sake of contradiction that $X \neq V(G)$.
Let $v \in V(G) \setminus X$.
For any vertex $x \in X$, we define $B_x(v) := \{ t \in V(T) : \dist_G(t, \gamma_X(x)) \le \dist_G(v, x) \}$.
The set $B_x(v)$ can be seen as the set of all possible extensions of $\gamma_X$ to $v$ that are nonexpansive with respect to $x$.
We show that $\bigcap_{x \in X} B_x(v) \neq \emptyset$.

To this end, we show that $T[B_x(v)]$ is connected for every $x \in X$, and that $B_x(v) \cap B_y(v) \neq \emptyset$ for all distinct $x,y \in X$.
By the Helly property this will suffice to obtain $\bigcap_{x \in X} B_x(v) \neq \emptyset$.

Let $x \in X$.
Since $T$ is isometric in $G$, we have $\dist_G(t,\gamma_X(x)) = \dist_T(t,\gamma_X(x))$ for every $t \in V(T)$.
Let $c_x = \gamma_X(x)$ and $r_x = \dist_G(v,x)$.
Then, $B_x(v) = \{t \in V(T) : \dist_G(t,c_x) \leq r_x\} = \ball_T(c_x,r_x)$ where $\ball_T(c,r)$ is the set ball of radius $r$ around $c$ in $T$.

Clearly, $T[\ball_T(c,r)]$ is connected for any $c \in V(T)$ and any $r \geq 0$.
Also, let $x,y \in X$ be distinct.
We have,
\[\dist_T(c_x,c_y) = \dist_G(\gamma_X(x),\gamma_X(y)) \leq \dist_G(x,y) \leq \dist_G(x,v) + \dist_G(v,y) = r_x + r_y.\]
It follows that $B_x(v) \cap B_y(v) = \ball_T(c_x,r_x) \cap \ball_T(c_y,r_y) \neq \emptyset$.
We conclude that the assumptions needed to apply the Helly property are satisfied and indeed $\bigcap_{x \in X} B_x(v) \neq \emptyset$ holds.

Let $X' = X \cup \{v\}$, and define $\gamma_{X'}(x) = \gamma_X(x)$ for every $x \in X$, and $\gamma_{X'}(v) = t$ for any $t \in \bigcap_{x \in X} B_x(v)$.
It is immediate to check that $\gamma_{X'}$ is a nonexpansive partial retraction onto $T$, contradicting the choice of $X$.
This completes the proof.
\end{proof}

\begin{lemma}\label{lem:retract2sep}
Let $k$ and $t$ be positive integers, and let $b$ be a positive real number.
Let $G$ be a graph, let $\{T_1,\dots,T_k\}$ be a family of trees forming an isometric vertex-cover of $G$, and let $U \subseteq V(G)$.
If any subtree of radius at most $t$ of any of these trees has at most $b$ vertices, then there exists $S \subset U$ of size at most $k|U|/t$ such that every component of $G[U] - S$ is of size at most $kb^k$.
\end{lemma}
\begin{proof}
    For every $i \in [k]$, let $\gamma_i$ be a nonexpansive retraction of $G$ onto $T_i$ (given by \cref{lem:retraction}). 
    Let $i \in [k]$.
    Let $r_i$ be an arbitrary vertex of $T_i$.
    For every nonnegative integer $h$, let
    \[L_i^h := \{u \in V(T_i) : \dist_{T_i}(r_i,u) = h\}.\]
    For every $s \in [t]$, also define
    $$
    A_i^s := \{ u \in U : \dist_{T_i}(r_i,\gamma_i(u)) \equiv s \bmod t\}.
    $$
    Clearly, the family $\{A_i^s : s \in [t]\}$ is a partition of $U$ (with some parts potentially empty). 
    In particular, there exists $s_i$ such that $|A_i^{s_i}|\leq |U|/t$.
    
     We define
     \[\textstyle S := \bigcup_{i \in [k]} A_i^{s_i}.\] 
     Note that $|S| \leq k|U|/t$. We prove that every component of $G[U] - S$ is of size at most $kb^k$. This will complete the proof.
    
    Consider the function $\Gamma: V(G) \to V(T_1) \times \dots \times V(T_k)$ such that $\Gamma(u) = (\gamma_1(u),\ldots,\gamma_k(u))$ for every $u \in V(G)$. 
    Since $T_1,\ldots,T_k$ vertex-cover $G$, for every $u \in V(G)$, there exists $i \in [k]$ such that $\gamma_i(u) = u$. 
    In particular, this means that for any $(t_1,\ldots,t_k) \in V(T_1)\times\ldots \times V(T_k)$, we have $\Gamma^{-1}(t_1,\ldots,t_k) \subseteq \{t_1, \ldots, t_k\}$, and so $|\Gamma^{-1}(t_1,\ldots,t_k)| \leq k$.

    Let $C$ be a component of $G[U] - S$.
    We have $C \subset \Gamma^{-1}(\gamma_1(V(C))\times\dots\times\gamma_k(V(C)))$.
    Let $i \in [k]$.
    By the definition of $A_i^{s_i}$, we have
    \[
    \gamma_i(V(C)) \cap \bigcup_{h \equiv s_i \bmod t} L_i^h = \emptyset.
    \]
    Thus, $\gamma_i(V(C))$ induces a connected subgraph of $T_i - \bigcup_{h \equiv s_i \bmod t} L_i^h$ by \cref{obs:retraction-connected}.
    It follows that $\gamma_i(V(C))$ induces a subgraph of radius at most $t$ in $T_i$.
    Thus, $|\gamma_i(V(C))| \leq b$.
    Summarizing, we obtain
    $$
       |C| \leq  |\Gamma^{-1}(\gamma_1(V(C))\times \ldots \times \gamma_k(V(C)))| \leq k \prod_{i \in [k]} |\gamma_i(V(C))| \leq kb^k,
    $$
    as desired.
\end{proof}

To conclude the proof of \cref{thm:positive_general,thm:positive_growth} we use the following result of Dvo\v{r}\'{a}k and Norin~\cite{dvovrak2019treewidth}.
Given an $n$-vertex graph $G$, a subset $S \subset V(G)$ is a \defin{balanced separator} if every component of $G - S$ has at most $2n/3$ vertices.

\begin{theorem}[Dvo\v{r}\'{a}k and Norin~\cite{dvovrak2019treewidth}]
\label{thm:sep_to_tw}
Let $G$ be a graph and let $a$ be a positive integer.
If every subgraph of $G$ has a balanced separator of size at most $a$, then $\tw(G) \leq 15a$.
\end{theorem}

\begin{proof}[Proof of \cref{thm:positive_general}]
The statement is clear when $\Delta = 1$, thus, we assume that $\Delta \geq 2$.
Let $G$ be an $n$-vertex graph with $\Delta(G) \leq \Delta$ admitting isometric vertex-covering with $k$ trees.
Let $U \subseteq V(G)$ and let $m = |U|$.
For a tree $T$ with $\Delta(T) \leq \Delta$ and a positive integer $t$, each subtree of $T$ of radius at most $t$ contains at most $\Delta^t + 1$ vertices.
For $m$ large enough, let $t$ be the largest positive integer such that 
\[
\Delta^t+1 \leq \left(\frac{2m}{3k}\right)^{1/k}.
\]
Then $t = \Omega_{k,\Delta}(\log m)$.
We apply \cref{lem:retract2sep} with $b = (2m/3k)^{1/k}$, obtaining $S \subset U$ with
\[
|S| \leq k m/t = O_{k,\Delta}\left(\frac{m}{\log m}\right)
\]
such that every component of $G[U]-S$ has size at most $k b^k = 2m/3$.
For bounded $m$ the same conclusion holds by taking $S=U$.
The set $S$ is a balanced separator of any subgraph of $G$ with the vertex set $U$.
\Cref{thm:sep_to_tw} implies the conclusion
\[
\tw(G) = O_{k,\Delta}\left(\frac{n}{\log n}\right). \qedhere
\]
\end{proof}

\begin{proof}[Proof of \cref{thm:positive_growth}]
Let $G$ be an $n$-vertex graph admitting isometric vertex-covering with $k$ trees, each of $(c,d)$-growth.
Let $U \subseteq V(G)$ and let $m = |U|$.
Every subtree of radius at most $t$ of a tree of $(c,d)$-growth has at most $ct^d$ vertices.
For $m$ large enough, let $t$ be the largest positive integer such that
\[
c t^d \leq \left(\frac{2m}{3k}\right)^{1/k}.
\]
Then $t = \Omega_{k,c,d}(m^{1/(kd)})$.
We apply \cref{lem:retract2sep} with $b = (2m/3k)^{1/k}$, obtaining $S \subset U$ with
\[
|S| \leq km/t = O_{k,c,d}\left(m^{1-1/(kd)}\right)
\]
such that every component of $G[U]-S$ has size at most $k b^k = 2m/3$.
For bounded $m$ the same conclusion holds by taking $S=U$.
The set $S$ is a balanced separator of any subgraph of $G$ with the vertex set $U$.
\Cref{thm:sep_to_tw} implies the conclusion
\[
\tw(G) = O_{k,c,d}\left(n^{1-1/(kd)}\right). \qedhere
\]
\end{proof}

\section{Constant subdivisions}\label{sec:constant-subdivisions}



In this section, we prove \cref{thm:constant-sibdivisions}.
Given a graph $G$ and a positive integer $d$, the \defin{$d$th power} of $G$, denoted by \defin{$G^d$} is a graph on the vertex set $V(G)$ where for all distinct $u,v \in V(G)$, they are adjacent if and only if $\dist_G(u,v) \leq d$.

\begin{lemma}\label{lem:reg-sets}
Let $f$ be an increasing function on nonnegative integers such that $f(r) \geq r$ for every nonnegative integer $r$.
Let $k$ be a positive integer.
There exists an integer $R_k$ such that, for every graph $G$, for every $v_0 \in V(G)$, and for any choice of sets $A_1, \ldots, A_k \subset V(G)$ covering $B_G(v_0,R_k)$,
there exist an integer $r$ with $0 \leq r \leq R_k$ and $v \in B_G(v_0,R_k)$ such that there is $i \in [k]$ where $A_i \cap B_G(v',r) \neq \emptyset$ for every $v' \in B_G(v,f(r))$.
\end{lemma}
\begin{proof}
The proof is by induction on $k$.
In the case $k = 1$, we set $R_k = f(0)$.
Then, $B_G(v_0,R_1) \subset A_1$ and we set $r := 0$ and $v := v_0$.
Next, we assume that the lemma holds for some $k \geq 1$ and we let $R_k$ be the witnessing radius.

Let $G$ be a graph and let $v_0 \in V(G)$.
We set $R_{k+1} := R_k + f(R_k)$.
Let $A_1, \ldots, A_{k+1} \subset V(G)$ be such that $B_G(v_0,R_{k+1}) \subset \bigcup_{i \in [k+1]} A_i$.
If $A_{k+1} \cap B_G(v',R_k) \neq \emptyset$ for every $v' \in B_G(v_0,f(R_k))$, then the lemma holds by taking $v := v_0$ and $r := R_k$.
Otherwise, there exists a vertex $v_1 \in B_G(v_0,f(R_k))$ such that the ball $B_G(v_1,R_k)$ is disjoint from $A_{k+1}$.
Since $B_G(v_1,R_k) \subseteq B_G(v_0,R_{k+1})$, the sets $A_1, \ldots, A_k$ cover $B_G(v_1,R_k)$.
By the induction hypothesis, there exists an integer $r$ with $0 \leq r \leq R_k$ and $v \in B_G(v_1,R_k)$ such that there is $i \in [k]$ where $A_i \cap B_G(v',r) \neq \emptyset$ for every $v' \in B_G(v,f(r))$.
This concludes the proof.
\end{proof}

\begin{lemma}\label{lem:wall-simplif}
Let $f$ be an increasing function on nonnegative integers such that $f(r) \geq r$ for every nonnegative integer $r$.
Let $k$ be a positive integer.
There exists an integer $N_k$ such that for every positive integer $s$, for every graph $G$, and for any choice of sets $A_1,\dots,A_k$ covering $V(G)$, if $G$ contains an ($\leq s$)-subdivision of the wall of order $N_k$ as a subgraph, then there exists an integer $r$ such that there is $i \in [k]$ such that with $d:= 16(s+1)(2r+2)$, $G^d[A_i]$ contains a square grid of order $\lfloor f(r) / d \rfloor$ as a subgraph.
\end{lemma}
\begin{proof}
Let $R_k$ be given by \cref{lem:reg-sets}.
We choose $N_k$ sufficiently large compared to $R_k+f(R_k)$.
Let $W$ be a subgraph of $G$ isomorphic to a $(\leq s)$-subdivision of the wall of order $N_k$.
We choose a vertex $v_0$ of $W$ sufficiently far from the boundary of the wall.

For every $i\in[k]$, let $A_i' := A_i\cap V(W)$.
The sets $A_1',\ldots,A_k'$ cover $V(W)$, and in particular they cover $B_W(v_0,R_k)$.
By \cref{lem:reg-sets}, there exist an integer $r$ with $0 \leq r \leq R_k$, a vertex $v\in V(W)$, and $i\in[k]$ such that $A_i'\cap B_W(v',r)\neq\emptyset$ for every $v'\in B_W(v,f(r))$.
We may assume that $f(r) \geq 8(s+1)(2r+2)$ as otherwise the statement is vacuous.

Since $f$ is increasing, $f(r)\leq f(R_k)$.
Thus, by the choice of $v_0$ and $N_k$, the ball $B_W(v,f(r))$ does not meet the boundary of the wall.
Set $q:=s+1$ and $\ell:=2r+2$.
By the definition of the wall, $W[B_W(v,f(r))]$ contains a subgraph $W'$ isomorphic to a $(\leq s)$-subdivision of a wall of order at least
$\left\lfloor(f(r)-q)/4q\right\rfloor \geq \left\lfloor f(r)/8q\right\rfloor$.
Indeed, every vertex of $W$ is at distance at most $s+1$ from a vertex corresponding to a vertex of the original wall, and distances in a $(\leq s)$-subdivision increase by a factor of at most $s+1$.

Let $X$ be a wall of order $p$.
It follows directly from the definition of a wall that $X^4$ contains a
$\left\lfloor\frac p2\right\rfloor \times \left\lfloor\frac p2\right\rfloor$
grid as a subgraph.
Consequently, $(W')^{4q}$ contains a square grid $Q$ of order at least $\left\lfloor f(r)/16q\right\rfloor$.
We choose the natural copy of this grid coming from the wall, so that if two vertices of $Q$ are at distance at least $a$ in $Q$, then their distance in $W$ is at least $a$.
From $Q$, keep every $\ell$-th row and every $\ell$-th column.
We obtain a square grid $Q'$ of order at least
$\left\lfloor f(r)/16q\ell\right\rfloor$,
whose distinct vertices are pairwise at distance at least $\ell=2r+2$ in $W$.

For every $x\in V(Q')$, since $x\in B_W(v,f(r))$, choose
$a(x)\in A_i'\cap B_W(x,r)$.
The vertices $a(x)$ are pairwise distinct.
Indeed, if $a(x)=a(y)$ for distinct $x,y\in V(Q')$, then
$\dist_W(x,y) \leq \dist_W(x,a(x))+\dist_W(a(y),y)\leq 2r$,
contrary to the choice of $Q'$.

Suppose now that $x$ and $y$ are adjacent in $Q'$.
They are joined in $Q$ by a path of length at most $\ell$, and every edge of $Q$ joins two vertices at distance at most $4q$ in $W$.
Hence $\dist_W(x,y)\leq 4q\ell$.
It follows that
\[
\dist_G(a(x),a(y)) \leq \dist_W(a(x),x)+\dist_W(x,y)+\dist_W(y,a(y)) \leq 2r+4q\ell \leq 16q\ell.
\]
Therefore, the vertices $a(x)$ for $x \in V(Q')$ induce in $G^{d}[A_i]$ a subgraph containing a square grid of order $\lfloor f(r) / d \rfloor$.
This completes the proof.
\end{proof}

\begin{proof}[Proof of \cref{thm:constant-sibdivisions}]
Let $k$, $d$, $t$, and $s$ be positive integers.
For every nonnegative integer $r$, set $\ell(r) := 16(s+1)(2r+2)$ and consider the function
\[
f(r) := \ell(r)\cdot \left((t+1)(d+1)^{\ell(r)}+1\right).
\]
Note that $f$ is increasing and $f(r) \geq r$ for every nonnegative integer $r$.
Let $g := N_k$, where $N_k$ is given by \cref{lem:wall-simplif} applied to $f$.

Suppose for the sake of contradiction that $G$ contains an ($\leq s$)-subdivision of the wall of order $g$, and that $\{H_1,\ldots,H_k\}$ is an isometric vertex-covering of $G$ such that $\Delta(H_i) \leq d$ and $\tw(H_i) \leq t$ for every $i \in [k]$.
Apply \cref{lem:wall-simplif} with $A_i := V(H_i)$ for every $i \in [k]$.
There exist a nonnegative integer $r$ and $i \in [k]$ such that $G^{\ell(r)}[V(H_i)]$ contains a square grid of order 
\[\left\lfloor \frac{f(r)}{\ell(r)} \right\rfloor = \left\lfloor\left((t+1)(d+1)^{\ell(r)}+1\right) \right\rfloor\]
as a subgraph.
We claim that 
\[\tw(G^{\ell(r)}[V(H_i)]) \leq (t+1)(d+1)^{\ell(r)}-1.\]
This will be a contradiction since treewidth of the square grid is equal to its order.
Since $H_i$ is isometric in $G$, we have $G^{\ell(r)}[V(H_i)] = H_i^{\ell(r)}$.
Let $(T,(W_x : x\in V(T)))$ be a tree decomposition of $H_i$ of width at most $t$.
For every $x\in V(T)$, define 
\[\textstyle W'_x := \bigcup_{v\in W_x} B_{H_i}(v,\ell(r)).\]
It is straightforward to check that $(T,(W'_x : x\in V(T)))$ is a tree decomposition of $H_i^{\ell(r)}$.
Moreover, since $\Delta(H_i)\leq d$, every ball of radius $\ell(r)$ in $H_i$ contains at most $(d+1)^{\ell(r)}$ vertices.
Therefore,
\[ |W'_x| \leq |W_x|(d+1)^{\ell(r)} \leq (t+1)(d+1)^{\ell(r)}\]
for every $x\in V(T)$, proving the claim and the theorem.
\end{proof}

\section{Open Problems}\label{sec:open}

We presented two ways of hiding a complex graph (usually a grid) inside a graph coverable by trees. The first is by the addition of well-chosen vertices of very large degree, the second by an extensive use of long subdivisions.
In addition, \Cref{sec:constant-subdivisions} shows that, within graphs of bounded degree, subdivisions can play a necessary role.
Our main concern is that this two ideas are significantly altering the \emph{metric} of the initial complex graph (usually a grid), making it almost trivial: in the first case the diameter of the graph becomes bounded, and in the second, its metric is ``diluted'' by the subdivisions.
In both cases, the asymptotic dimension of the class is upper bounded by $1$, when it was initially $\ge 2$. Hence the following question.

\begin{question}
Is there a class of graphs of asymptotic dimension at least 2, coverable by a bounded number of isometric trees? 
\end{question}

Another interesting question is whether \Cref{thm:positive_general} is tight.

\begin{question}
Is there a constant $c$ such that there is, for arbitrarily large $n$, an $n$-vertex graph of treewidth $\Omega(n / \log n)$ coverable by $c$ isometric trees?
\end{question}

\section*{Acknowledgments}
The research was conducted during the 2025 Sino-French International Conference on Graph Theory, Combinatorics, and Algorithms, hosted by Shandong University in Jinan, China.
We are grateful to the organizers and participants for creating a friendly and stimulating environment.

\section*{Statement of AI use}
The counter-example constructed in the proof of \Cref{thm:bd-degree} was inspired by interactions with OpenAI's ChatGPT-5.6 Sol Pro.
The authors take responsibility for the mathematical correctness of the presented arguments.

\bibliographystyle{plain}
\bibliography{bibs}
\end{document}